\documentclass[12pt,twoside]{amsart}
\usepackage{amsmath, amsthm, amscd, amsfonts, amssymb, graphicx}
\usepackage{enumerate}
\usepackage{mathrsfs}
\newtheorem{theorem}{Theorem}[section]
\newtheorem{lemma}{Lemma}[section]
\newtheorem{remark}{Remark}[section]

\newtheorem{corollary}{Corollary}[section]

\newtheorem{proposition}{Proposition}[section]
\begin{document}
	\title{A glimpse at the operator Kantorovich inequality}
	\author{Hamid Reza Moradi$^{\dagger}$, IBRAHIM HALIL G\"uM\"U\c S$^{\ddagger}$ and Zahra Heydarbeygi$^{\S}$}
	
		 \maketitle
	\begin{abstract}
We show the following result: Let $A$ be a positive operator satisfying $0<m{{\mathbf{1}}_{\mathcal{H}}}\le A\le M{{\mathbf{1}}_{\mathcal{H}}}$  for some scalars $m,M$ with $m<M$ and $\Phi $ be a normalized positive linear map, then 
\[\Phi \left( {{A}^{-1}} \right)\le \Phi \left( {{m}^{\frac{A-M{{\mathbf{1}}_{\mathcal{H}}}}{M-m}}}{{M}^{\frac{m{{\mathbf{1}}_{\mathcal{H}}}-A}{M-m}}} \right)\le \frac{{{\left( M+m \right)}^{2}}}{4Mm}\Phi {{\left( A \right)}^{-1}}.\]
Besides, we prove that the second inequality in the above can be squared.
	\end{abstract}
\pagestyle{myheadings}
\markboth{\centerline {A glimpse at the operator Kantorovich inequality}}
{\centerline {H.R. Moradi, I.H. G\"um\"u\c s \& Z. Heydarbeygi}}
\bigskip
\bigskip
\section{\bf Introduction}
\vskip 0.4 true cm
In 1948, L.V. Kantorovich,  Soviet mathematician and economist, introduced the well-known Kantorovich inequality \cite{kantorovich}. Operator version of Kantorovich inequality was firstly established by A.W. Marshall and I. Olkin, who obtained:

\vskip 0.4 true cm
\noindent{\bf Theorem A.}
{\it 
	{\upshape(\cite{marshal})} Let $A$ be a positive operator satisfying $0<m{{\mathbf{1}}_{\mathcal{H}}}\le A\le M{{\mathbf{1}}_{\mathcal{H}}}$ for some scalars $m,M$ with $m<M$ and $\Phi $ be a normalized positive linear map. Then
	\begin{equation}\label{5}
	\Phi \left( {{A}^{-1}} \right)\le \frac{{{\left( M+m \right)}^{2}}}{4Mm}\Phi {{\left( A \right)}^{-1}}.
	\end{equation}
}
This note aims to present an improvement of inequality \eqref{5}. The main result of this note is of this genre:
\begin{theorem}\label{th1}
	Let all the assumptions of Theorem {\upshape A} hold. Then 
\begin{equation}\label{10}
\Phi \left( {{A}^{-1}} \right)\le \Phi \left( {{m}^{\frac{A-M{{\mathbf{1}}_{\mathcal{H}}}}{M-m}}}{{M}^{\frac{m{{\mathbf{1}}_{\mathcal{H}}}-A}{M-m}}} \right)\le \frac{{{\left( M+m \right)}^{2}}}{4Mm}\Phi {{\left( A \right)}^{-1}}.
\end{equation}	
\end{theorem}
This is proven at the end of Section \ref{s2}. We start off by fixing some notation: Let $\mathbb{B}\left( \mathcal{H} \right)$ denote the set of all bounded linear operators on a complex Hilbert space $\mathcal{H}$ with the identity ${{\mathbf{1}}_{\mathcal{H}}}$. We extensively use the continuous functional calculus for self-adjoint operators, e.g., see \cite[p. 3]{book}. An operator $A$ on $\mathcal{H}$ is said to be {\it positive} (in symbol $0\le A$) if $0\le \left\langle Ax,x \right\rangle $ for all $x\in \mathcal{H}$. We write $0<A$ if $A$ is positive and invertible. For self-adjoint operators $A,B\in \mathbb{B}\left( \mathcal{H} \right)$, we say $A\le B$ if $0\le B-A$. A linear map $\Phi :\mathbb{B}\left( \mathcal{H} \right)\to \mathbb{B}\left( \mathcal{K} \right)$, where $\mathcal{H}$ and $\mathcal{K}$ are complex Hilbert spaces, is called {\it positive} if $\Phi \left( A \right)\ge 0$ whenever $A\ge 0$ and is said to be {\it normalized} if $\Phi \left( {{\mathbf{1}}_{\mathcal{H}}} \right)={{\mathbf{1}}_{\mathcal{K}}}$.

A positive function defined on the interval $I$ (or, more generally, on a convex subset of some vector space) is called  {\it $log $-convex} if $\log f\left( x \right)$ is a convex function of $x$. We observe that such functions satisfy the elementary inequality
\begin{equation*}
f\left( \left( 1-v \right)a+vb \right)\le {{\left[ f\left( a \right) \right]}^{1-v}}{{\left[ f\left( b \right) \right]}^{v}},\qquad \text{ }0\le v\le 1
\end{equation*}
for any $a,b\in I$. Because of the weighted arithmetic-geometric mean inequality, we also have
\begin{equation}\label{b}
f\left( \left( 1-v \right)a+vb \right)\le {{\left[ f\left( a \right) \right]}^{1-v}}{{\left[ f\left( b \right) \right]}^{v}}\le \left( 1-v \right)f\left( a \right)+vf\left( b \right),
\end{equation}
which says that any log-convex function is a convex function. 

\medskip

The following inequality is well known in the literature as the Choi-Davis-Jensen inequality:
\vskip 0.4 true cm

\noindent{\bf Theorem B.}
(\cite{choi, davis}) {\it Let $A\in \mathcal{B}\left( \mathcal{H} \right)$ be a self-adjoint operator with spectrum $Sp\left( A \right)\subseteq I$ and $\Phi $ be a normalized positive linear map from $\mathbb{B}\left( \mathcal{H} \right)$ to $\mathbb{B}\left( \mathcal{K} \right)$. If $f$ is operator convex function on an interval $I$, then
\begin{equation}\label{cdj}
f\left( \Phi \left( A \right) \right)\le \Phi \left( f\left( A \right) \right).
\end{equation}
}
\indent Though in the case of convex function the inequality \eqref{cdj} does not hold in general, we have the following estimate:
\vskip 0.4 true cm
\noindent{\bf Theorem C.}
{\it {\upshape(\cite[Remark 4.14]{micic})} Let $A\in \mathcal{B}\left( \mathcal{H} \right)$ be a self-adjoint operator with $Sp\left( A \right)\subseteq \left[ m,M \right]$ for some scalars $m,M$ with $m<M$ and $\Phi $ be a normalized positive linear map from $\mathbb{B}\left( \mathcal{H} \right)$ to $\mathbb{B}\left( \mathcal{K} \right)$. If $f$ is non-negative convex function, then
\begin{equation*}
\frac{1}{\mu \left( m,M,f \right)}\Phi \left( f\left( A \right) \right)\le f\left( \Phi \left( A \right) \right)\le \mu \left( m,M,f \right)\Phi \left( f\left( A \right) \right),
\end{equation*}
where $\mu\left( m,M,f \right)$ is defined by
\[\mu\left( m,M,f \right)\equiv \max \left\{ \frac{1}{f\left( t \right)}\left( \frac{M-t}{M-m}f\left( m \right)+\frac{t-m}{M-m}f\left( M \right) \right):\text{ }m\le t\le M \right\}.\]
}

In Section \ref{s2} we prove an analogue of Theorem C for log-convex functions. The proof of Theorem \ref{th1} follows quickly from this inequality. In Section \ref{s3}, inspired by the work of Lin \cite{lin}, we square the second inequality in \eqref{10}.

\section{\bf A refinement of the operator Kantorovich inequality}\label{s2}
\vskip 0.4 true cm
An important role in our analysis is played by the following result, which is of independent interest. 
\begin{proposition}\label{thb}
Let all the assumptions of Theorem {\upshape C} hold except the condition convexity which is changed to log-convexity. Then
\begin{equation}\label{1}
\Phi \left( f\left( A \right) \right)\le \Phi \left( {{\left[ f\left( m \right) \right]}^{\frac{M{{\mathbf{1}}_{\mathcal{H}}}-A}{M-m}}}{{\left[ f\left( M \right) \right]}^{\frac{A-m{{\mathbf{1}}_{\mathcal{H}}}}{M-m}}} \right)\le \mu \left( m,M,f \right)f\left( \Phi \left( A \right) \right).
\end{equation}
\end{proposition}
\begin{proof}
It can be verified that if $m\le t\le M$, then $0\le \frac{M-t}{M-m},\frac{t-m}{M-m}\le 1$ and $\frac{M-t}{M-m}+\frac{t-m}{M-m}=1$. Thanks to \eqref{b}, we have 
\begin{equation}\label{c}
f\left( t \right)=f\left( \frac{M-t}{M-m}m+\frac{t-m}{M-m}M \right)\le {{\left[ f\left( m \right) \right]}^{\frac{M-t}{M-m}}}{{\left[ f\left( M \right) \right]}^{\frac{t-m}{M-m}}}\le L\left( t \right),
\end{equation}
where
\[L\left( t \right)=\frac{M-t}{M-m}f\left( m \right)+\frac{t-m}{M-m}f\left( M \right).\]
Applying functional calculus for the operator $A$, we infer that
\[f\left( A \right)\le {{\left[ f\left( m \right) \right]}^{\frac{M{{\mathbf{1}}_{\mathcal{H}}}-A}{M-m}}}{{\left[ f\left( M \right) \right]}^{\frac{A-m{{\mathbf{1}}_{\mathcal{H}}}}{M-m}}}\le L\left( A \right).\]
Using  the  hypotheses  made  about $\Phi $,
\begin{equation}\label{3}
\Phi \left( f\left( A \right) \right)\le \Phi \left( {{\left[ f\left( m \right) \right]}^{\frac{M{{\mathbf{1}}_{\mathcal{H}}}-A}{M-m}}}{{\left[ f\left( M \right) \right]}^{\frac{A-m{{\mathbf{1}}_{\mathcal{H}}}}{M-m}}} \right)\le \Phi \left( L\left( A \right) \right).
\end{equation}
On account of \cite[Corollary 4.12]{micic} (the functions $f$ and $g$ there are now $L$ and $f$, respectively), we get
\[\Phi \left( f\left( A \right) \right)\le \Phi \left( {{\left[ f\left( m \right) \right]}^{\frac{M{{\mathbf{1}}_{\mathcal{H}}}-A}{M-m}}}{{\left[ f\left( M \right) \right]}^{\frac{A-m{{\mathbf{1}}_{\mathcal{H}}}}{M-m}}} \right)\le \mu \left( m,M,f \right)f\left( \Phi \left( A \right) \right).\]
Notice that, although \cite[Corollary 4.12]{micic} is for matrices, it is also true for operators.

Hence \eqref{1} follows.
\end{proof}

\medskip

The following follows immediately from Proposition \ref{thb}. Recall that $f\left( t \right)={{t}^{p}}$, $\left( p<0 \right)$ is log-convex function.
\begin{corollary}\label{3.1}
Under the hypotheses of Proposition \ref{thb}, let $p\in \left( -\infty ,0 \right)$ and $0<m<M$. Then 
\begin{equation}\label{18}
\Phi \left( {{A}^{p}} \right)\le \Phi \left( {{m}^{p\left( \frac{M{{\mathbf{1}}_{\mathcal{H}}}-A}{M-m} \right)}}{{M}^{p\left( \frac{A-m{{\mathbf{1}}_{\mathcal{H}}}}{M-m} \right)}} \right)\le K\left( m,M,p \right)\Phi {{\left( A \right)}^{p}},
\end{equation}
where $K\left( m,M,p \right)$ is the generalized Kantorovich constant defined by
\[K\left( m,M,p \right)\equiv \frac{m{{M}^{p}}-M{{m}^{p}}}{\left( p-1 \right)\left( M-m \right)}{{\left( \frac{p-1}{p}\frac{{{M}^{p}}-{{m}^{p}}}{m{{M}^{p}}-M{{m}^{p}}} \right)}^{p}}.\]
\end{corollary}
\begin{remark}
We would like to mention that \eqref{1} can be regarded as an improvement of \cite[Theorem 1.5]{f} (see also \cite[Lemma 2]{micic1}).
\end{remark}

\medskip

After the previous technical intermission, we return to the main subject of this section,
the proof of the inequality \eqref{10}.

\medskip

\noindent {\it Proof of Theorem \ref{th1}.} This follows from Corollary \ref{3.1} by putting $p=-1$. We should point out that $K\left( m,M,-1 \right)=\frac{{{\left( M+m \right)}^{2}}}{4Mm}$. $\hfill\square$

\vskip 0.5 true cm

Can the second inequality in \eqref{10} be squared? Responding to this question is the main motivation of the next section.
\section{\bf Squaring refinement of the operator Kantorovich inequality}\label{s3}
\vskip 0.4 true cm
We will need the following lemmas.
\begin{lemma}\label{11}
		\hfill
\begin{itemize}
	\item[(i)] \cite[Theorem 1]{bhatia} Let $A,B>0$. Then the following norm inequality holds:
		\[\left\| AB \right\|\le \frac{1}{4}{{\left\| A+B \right\|}^{2}}.\] 
	\item[(ii)] \cite[Theorem 3]{ando} Let $A,B\ge 0$ and $1\le r\le \infty $. Then
	\[\left\| {{A}^{r}}+{{B}^{r}} \right\|\le \left\| {{\left( A+B \right)}^{r}} \right\|.\]
\end{itemize}
\end{lemma}
\begin{lemma}\label{9}
For each $m\le t\le M$, we have
\[t+mM{{m}^{\frac{t-M}{M-m}}}{{M}^{\frac{m-t}{M-m}}}\le M+m.\] 
\end{lemma}
\begin{proof}
Because of the weighted arithmetic-geometric mean inequality
\[\begin{aligned}
 t+mM{{m}^{\frac{t-M}{M-m}}}{{M}^{\frac{m-t}{M-m}}}&=t+{{m}^{\frac{t-m}{M-m}}}{{M}^{\frac{M-t}{M-m}}} \\ 
& \le t+\frac{t-m}{M-m}m+\frac{M-t}{M-m}M \\ 
& =M+m,  
\end{aligned}\]
which finishes the proof.
\end{proof}

\medskip

Now we are at the position to state our main result.
\begin{theorem}\label{th3}
Let all the assumptions of Theorem {\upshape A} hold. Then
\begin{equation}\label{12}
\Phi {{\left( {{m}^{\frac{A-M{{\mathbf{1}}_{\mathcal{H}}}}{M-m}}}{{M}^{\frac{m{{\mathbf{1}}_{\mathcal{H}}}-A}{M-m}}} \right)}^{p}}\le {{\left( \frac{{{\left( M+m \right)}^{2}}}{{{4}^{\frac{2}{p}}}Mm} \right)}^{p}}\Phi {{\left( A \right)}^{-p}}\quad\text{ for }2\le p<\infty .
\end{equation}
In particular
\[\Phi {{\left( {{m}^{\frac{A-M{{\mathbf{1}}_{\mathcal{H}}}}{M-m}}}{{M}^{\frac{m{{\mathbf{1}}_{\mathcal{H}}}-A}{M-m}}} \right)}^{2}}\le {{\left( \frac{{{\left( M+m \right)}^{2}}}{4Mm} \right)}^{2}}\Phi {{\left( A \right)}^{-2}}.\]
\end{theorem}
\begin{proof}
The idea of the proof is similar to \cite[Theorem 3]{fu}. It is easy to see that if $A,B>0$ and $\alpha >0$, then
\[A\le \alpha B\quad\text{ }\Leftrightarrow \quad\text{ }\left\| {{A}^{\frac{1}{2}}}{{B}^{-\frac{1}{2}}} \right\|\le {{\alpha }^{\frac{1}{2}}}.\] 
So we are done if we can show 
\[\left\| \Phi {{\left( {{m}^{\frac{A-M{{\mathbf{1}}_{\mathcal{H}}}}{M-m}}}{{M}^{\frac{m{{\mathbf{1}}_{\mathcal{H}}}-A}{M-m}}} \right)}^{\frac{p}{2}}}\Phi {{\left( A \right)}^{\frac{p}{2}}} \right\|\le \frac{{{\left( M+m \right)}^{p}}}{4{{M}^{\frac{p}{2}}}{{m}^{\frac{p}{2}}}}.\]
On account of Lemma \ref{9}, it follows that
\begin{equation}\label{6}
\Phi \left( A \right)+mM\Phi \left( {{m}^{\frac{A-M{{\mathbf{1}}_{\mathcal{H}}}}{M-m}}}{{M}^{\frac{m{{\mathbf{1}}_{\mathcal{H}}}-A}{M-m}}} \right)\le \left( M+m \right){{\mathbf{1}}_{\mathcal{H}}}.
\end{equation}
By direct calculation,
\[\begin{aligned}
& \left\| {{m}^{\frac{p}{2}}}{{M}^{\frac{p}{2}}}\Phi {{\left( {{m}^{\frac{A-M{{\mathbf{1}}_{\mathcal{H}}}}{M-m}}}{{M}^{\frac{m{{\mathbf{1}}_{\mathcal{H}}}-A}{M-m}}} \right)}^{\frac{p}{2}}}\Phi {{\left( A \right)}^{\frac{p}{2}}} \right\| \\ 
&\quad \le \frac{1}{4}{{\left\| {{m}^{\frac{p}{2}}}{{M}^{\frac{p}{2}}}\Phi {{\left( {{m}^{\frac{A-M{{\mathbf{1}}_{\mathcal{H}}}}{M-m}}}{{M}^{\frac{m{{\mathbf{1}}_{\mathcal{H}}}-A}{M-m}}} \right)}^{\frac{p}{2}}}+\Phi {{\left( A \right)}^{\frac{p}{2}}} \right\|}^{2}} \quad \text{(by Lemma \ref{11} (i))}\\ 
&\quad \le \frac{1}{4}{{\left\| {{\left( mM\Phi \left( {{m}^{\frac{A-M{{\mathbf{1}}_{\mathcal{H}}}}{M-m}}}{{M}^{\frac{m{{\mathbf{1}}_{\mathcal{H}}}-A}{M-m}}} \right)+\Phi \left( A \right) \right)}^{\frac{p}{2}}} \right\|}^{2}} \quad \text{(by Lemma \ref{11} (ii))}\\ 
&\quad =\frac{1}{4}{{\left\| mM\Phi \left( {{m}^{\frac{A-M{{\mathbf{1}}_{\mathcal{H}}}}{M-m}}}{{M}^{\frac{m{{\mathbf{1}}_{\mathcal{H}}}-A}{M-m}}} \right)+\Phi \left( A \right) \right\|}^{p}} \\ 
&\quad \le \frac{{{\left( M+m \right)}^{p}}}{4} \quad \text{(by \eqref{6})}. \\ 
\end{aligned}\]
This proves \eqref{12}.
\end{proof}

\medskip

{\bf Acknowledgements.} We would like to thank the referee(s) for carefully reading our manuscript and for giving such constructive comments which substantially helped to improve the quality of the paper.
\bibliographystyle{alpha}

\begin{thebibliography}{9}
\bibitem{ando}
T. Ando, X. Zhan, {\it Norm inequalities related to operator monotone functions}, Math. Ann. {\bf315} (1999), 771--780.

\bibitem{bhatia}
R. Bhatia, F. Kittaneh, {\it Notes on matrix arithmetic–geometric mean inequalities}, Linear Algebra Appl. {\bf308} (2000), 203--211.

\bibitem{choi}
M.D. Choi, {\it A Schwarz inequality for positive linear maps on C$^*$-algebras}, Illinois J. Math. {\bf18} (1974), 565--574.

\bibitem{davis}
C. Davis, {\it A Schwarz inequality for convex operator functions}, Proc. Amer. Math. Soc. {\bf8} (1957), 42--44.

\bibitem{fu}
X. Fu, C. He, {\it Some operator inequalities for positive linear maps}, Linear Multilinear Algebra. {\bf63}(3) (2015), 571--577.

\bibitem{book}
T. Furuta, J. Mi\'ci\'c-Hot, J. Pe\v cari\'c, Y. Seo, {\it Mond-Pe\v cari\'c method in operator inequalities}, Element, Zagreb, 2005.

\bibitem{f}
T. Furuta, {\it Operator inequalities associated with H\"older-McCarthy and Kantorovich inequalities}, J. Inequal. Appl. 1998.2 (1998): 234521.

\bibitem{kantorovich}
L.V. Kantorovi\'c, {\it Functional analysis and applied mathematics}, Uspekhi Math.
Nauk. {\bf3} (1948), Translated from the Russian by Curtis D. Benster, Nat. Bur. Stand
ards, Report No. 1509., March 7, 1952.

\bibitem{lin}
M. Lin, {\it On an operator Kantorovich inequality for positive linear maps}, J. Math. Anal.
Appl. {\bf402} (2013), 127--132.

\bibitem{marshal}
A.W. Marshall, I. Olkin, {\it Matrix versions of Cauchy and Kantorovich inequalities}, Aequationes Math. {\bf40} (1990), 89--93.

\bibitem{micic1}
J. Mi\'ci\'c, J. Pe\v cari\'c,  {\it Order among power means of positive operators, II},  Sci. Math. Japon. {\bf71}(1) (2010), 93--109.

\bibitem{micic}
J. Mi\'ci\'c, J. Pe\v cari\'c, Y. Seo, M. Tominaga, {\it Inequalities for positive linear maps on Hermitian matrices}, Math. Inequal. Appl. {\bf3} (2000), 559--591.
\end{thebibliography}

\vskip 0.4 true cm

\tiny$^{\dagger}$Young Researchers and Elite Club, Mashhad Branch, Islamic Azad University, Mashhad, Iran.

{\it E-mail address:} hrmoradi@mshdiau.ac.ir

\vskip 0.4 true cm

$^{\ddagger}$Department of Mathematics, Faculty of Arts and Sciences, Ad\i yaman University, Ad\i yaman,
Turkey.

{\it E-mail address:} igumus@adiyaman.edu.tr

\vskip 0.4 true cm

$^{\S}$Department of Mathematics, Mashhad Branch, Islamic Azad University, Mashhad, Iran.

{\it E-mail address:} zheydarbeygi@yahoo.com
\end{document}